\documentclass[reqno,11pt]{amsart}

\usepackage{latexsym}
\usepackage[margin=3cm]{geometry}
\usepackage{amscd}
\usepackage{amsthm}
\usepackage{marginnote}
\usepackage{subfigure,amsmath,amsfonts,amssymb, epsfig}
\usepackage{graphics}
\usepackage{xy}
\usepackage{mathrsfs}

\usepackage[latin1]{inputenc}

\usepackage{subfigure,amsmath}

\input xy
\xyoption{all}

\newtheorem{theorem}{Theorem}[section]

\newtheorem{proposition}[theorem]{Proposition}

\newtheorem{definition}[theorem]{Definition}

\theoremstyle{remark}
\newtheorem{remark}[theorem]{Remark}

\numberwithin{equation}{section}

\allowdisplaybreaks

\author[Michael J.\ Schlosser]{Michael J.\ Schlosser$^*$}
\address{Fakult\"at f\"ur Mathematik, Universit\"at Wien,
Oskar-Morgenstern-Platz~1, A-1090 Vienna, Austria}
\email{michael.schlosser@univie.ac.at}
\urladdr{http://www.mat.univie.ac.at/{\textasciitilde}schlosse}
\thanks{$^*$Partly supported by FWF Austrian Science Fund
grant F50-08 within the SFB
``Algorithmic and enumerative combinatorics''.}

\author[Meesue Yoo]{Meesue Yoo$^{**}$}
\address{Fakult\"at f\"ur Mathematik, Universit\"at Wien,
Oskar-Morgenstern-Platz~1, A-1090 Vienna, Austria}
\email{meesue.yoo@univie.ac.at}
\thanks{$^{**}$Fully supported by FWF Austrian Science Fund
grant F50-08 within the SFB
``Algorithmic and enumerative combinatorics''.}

\title{Basic hypergeometric summations from rook theory}
\dedicatory{Dedicated to Professor Krishnaswami Alladi
on the occasion of his 60th birthday}



\newcommand{\N}{\mathbb N}

\begin{document}

\begin{abstract}
We employ a one-variable extension of $q$-rook theory to give
combinatorial proofs of some basic hypergeometric summations,
including the $q$-Pfaff--Saalsch\"utz summation and
a ${}_4\phi_3$ summation by Jain.
\end{abstract}
 
 
\maketitle

\section{Introduction}

The theory of $q$-series has a prominent history.
It made its first appearance in a combinatorial
study by Euler on partitions of numbers.
Among the first $q$-series identities were explicit summations for 
two different $q$-analogues of the exponential series.
These identities were later unified by Gau{\ss}, Heine,
and Cauchy who, all three independent from each other,
discovered and proved the nonterminating $q$-binomial theorem.
This initiated the systematic study
of $q$-hypergeometric series, or synonymously,
basic hypergeometric series, ``basic'' referring to the {\em base} $q$,
as objects of their own interest, separate from combinatorics.
While in the early days only a small number of mathematicians studied
the combinatorics of $q$-series (most notably, J.~J.~Sylvester
in the 19th century, and P.~A.~MacMahon and I.~Schur in the early
20th century, to name just a few figures whose research had big impact)
the situation rapidly changed in the 1960s when B.~Gordon found
extensions of the Rogers--Ramanujan identities with combinatorial
interpretations, after which many more people entered the scene.
See Andrews' book chapter~\cite{A} for an account of
the history of $q$-series and partitions.
Further, the preface of Gasper and Rahman's textbook
\cite{GR} provides a brief history of basic hypergeometric
series, and the book itself contains further background on the subject.

Basic hypergeometric series appear from time to time in
combinatorial studies. It is particularly instructive to see
combinatorial proofs of $q$-series identities. Having a combinatorial
interpretation of an identity at hand leads to a better understanding,
since one gets a feeling why the identity is true.
Now, focusing on combinatorial interpretations
and given a reasonably simple identity,
it is by all means legitimate to ask the question:
is there a combinatorial proof for it?
For instance, for the well-known $q$-Pfaff--Saalsch\"utz summation,
\begin{equation}\label{32id}
\sum_{k=0}^n\frac{(a,b,q^{-n};q)_k}{(q,c,abq^{1-n}/c;q)_k}q^k=
\frac{(c/a,c/b;q)_n}{(c,c/ab;q)_n},
\end{equation}
(see Section~\ref{sec:defs} for the notation)
several combinatorial proofs are known \cite{AB,G,GZ,Schl0,Y,Z}.
But what about a combinatorial proof of the following
summation by Jain~\cite{J}
(which is a $q$-analogue of a ${}_3F_2$ summation by Bailey~\cite{B}
and can be written as a summation for a specific ${}_4\phi_3$ series)?
\begin{equation}\label{43id}
\sum_{k=0}^n \frac{(a,b;q)_k (q^{-2n};q^2)_k}
{(q,q^{-2n};q)_k(abq;q^2)_k}q^k=
\frac{(aq,bq;q^2)_n}{(q,abq;q^2)_n}.
\end{equation}
In this paper we present a new combinatorial proof of \eqref{32id}
and also, to the best of our knowledge,
a first combinatorial proof of \eqref{43id}, in addition to
a few similar results.
We achieve this by employing a specific one-variable extension
of $q$-rook theory with extra variable $a$, which we shall refer
to as $(a;q)$-rook theory for short.
(The $(a;q)$-rook numbers we are dealing with in this paper
are actually special cases of the elliptic rook numbers which
we recently have considered in \cite{SY0,SY1,SY2}. 
While the elliptic rook numbers satisfy nice identities,
they don't factorize in general. However, for particular boards the $(a;q)$-rook
numbers do factorize into closed form. This is the reason why
we focus  on the $(a;q)$-case here which is still more general
than the $q$-case.)

Already earlier, Haglund~\cite{H} has made out intimate
connections between rook theory and (basic) hypergeometric series.
In particular, he showed that a big class of
($q$-)rook numbers generally admit a representation in terms
of (basic) hypergeometric series of Karlsson--Minton type.
In our case, we are on one hand (for three different rook models)
working with $(a;q)$-rook numbers,
i.e., we add an extra parameter to the $q$-rook numbers.
On the other hand, we are looking at very special situations,
obtained by restricting to special boards,
where the $(a;q)$-rook numbers nicely factorize.
Since the $(a;q)$-rook numbers 
satisfy certain product formulas,
we are thus able to obtain explicit summations, by substituting
 the factorized forms in the product formulas.

In Section~\ref{sec:defs} we recall standard $q$-series notation
and introduce the special $(a;q)$-weights that we use.
Section~\ref{sec:rooks} is devoted to $(a;q)$-rook theory.
We explain all the ingredients we need for the $(a;q)$-extensions
of the different rook models that we work with, namely,
the standard model, the (more general) alpha-parameter model,
and the matching model. From these we deduce
basic hypergeometric summations as applications.

\section{Standard $q$-notation and $(a;q)$-weights}\label{sec:defs}


For a parameter $q$, called the \emph{base}, and variable $u$,
the $q$-shifted factorial is defined by
\begin{equation*}
(u;q)_0=1,\qquad\text{and}\qquad(u;q)_n=(1-u)(1-uq)\dots(1-uq^{n-1}).
\end{equation*}
(The index $n$ can also be $\infty$, then the product is an infinite product
in which case one requires $|q|<1$, for convergence.)
For brevity, we frequently use the notation
\begin{equation*}
(a_1,\dots,a_m;q)_n=(a_1;q)_n\dots(a_m;q)_n.
\end{equation*}
The $q$-number of $z$ is defined as
\begin{equation*}
[z]_q=\frac{1-q^z}{1-q}.
\end{equation*}

We now introduce {\em $(a;q)$-weights}
which include an additional variable $a$. Define 
\begin{equation*}
w_{a;q}(k)=\frac{(1-a q^{2k+1})}{(1-a q^{2k-1})}q^{-1}, 
\end{equation*}
\begin{equation*}
W_{a;q}(k)=\frac{(1-a q^{2k+1})}{(1-aq)}q^{-k}, 
\end{equation*}
\begin{equation*}
[z]_{a;q}=\frac{(1-q^z)(1-a q^z)}{(1-q)(1-aq)}q^{1-z}, 
\end{equation*}
for any value $k$, which we call the \emph{small} weights,
\emph{big} weights, and the \emph{$(a;q)$-number} of $z$, respectively.
Note that in the limit case $a\to \infty$, we recover
the $q$-weights
\begin{equation*}
\lim_{a\to \infty}w_{a;q}(k)=q,\quad \lim_{a\to \infty}W_{a;q}(k)=
q^k,\quad \lim_{a\to \infty}[z]_{a;q}=\frac{1-q^z}{1-q}=[z]_q.
\end{equation*}
For a positive integer $k$, we have 
\begin{equation*}
W_{a;q}(k)=\prod_{i=1}^k w_{a;q}(i).
\end{equation*}
Other useful properties are 
\begin{equation*}
[y+z]_{a;q} = [y]_{a;q}+W_{a;q}(y)[z]_{aq^{2y};q}, 
\end{equation*}
and 
\begin{equation*}
W_{a;q}(k+n)=W_{a;q}(k)W_{aq^{2k};q}(n). 
\end{equation*}

\begin{remark}
This $(a;q)$-weight was first defined in \cite{Schl1} to
generalize the binomial theorem for noncommuting variables.
That is, in the unital algebra $\mathbb{C}_{a;q}[x,y]$ over $\mathbb C$
defined by the following commutation relations 
\begin{align*}
yx&=\frac{(1-a q^3)}{(1-aq)}q^{-1}xy,\\
xa&=qax,\\ 
ya&=q^2 ay,
\end{align*} 
the binomial theorem
\begin{equation*}
(x+y)^n =\sum_{k=0}^n \begin{bmatrix}n\\k\end{bmatrix}_{a;q}x^k y^{n-k} 
\end{equation*}
holds, where the $(a;q)$-binomial coefficients are defined by 
\begin{equation}\label{aqbin}
\begin{bmatrix}n\\k\end{bmatrix}_{a;q}:=
\frac{(q^{1+k},aq^{1+k};q)_{n-k}}{(q,aq;q)_{n-k}}q^{k(k-n)}
=\frac{[n]_{a;q}!}{[k]_{a;q}![n-k]_{a;q}!},
\end{equation}
with the $(a;q)$-factorials being defined by 
\begin{equation*}
[0]_{a;q}!=1,\qquad\qquad[n]_{a,q}!=[n]_{a,q}[n-1]_{a,q}!.
\end{equation*}
The $(a;q)$-binomial coefficients are symmetric in $(k,n-k)$
(whereas the more general elliptic extension of \eqref{aqbin}, considered
in \cite{Schl0,Schl1}, is not). They satisfy
the two recursions
\begin{align*}
\begin{bmatrix}n+1\\k\end{bmatrix}_{a;q}&=
\begin{bmatrix}n\\k\end{bmatrix}_{a;q}+
\frac{(1-aq^{2n+2-k})}{(1-aq^k)}q^{k-n-1}\begin{bmatrix}n\\k-1\end{bmatrix}_{a;q},\\
\begin{bmatrix}n+1\\k\end{bmatrix}_{a;q}&=
\frac{(1-aq^{n+1+k})}{(1-aq^{n+1-k})}q^{-k}\begin{bmatrix}n\\k\end{bmatrix}_{a;q}+
\begin{bmatrix}n\\k-1\end{bmatrix}_{a;q},
\end{align*}
which, together with the initial conditions
\begin{equation*}
\begin{bmatrix}0\\0\end{bmatrix}_{a;q}=1,\qquad\text{and}\quad
\begin{bmatrix}n\\k\end{bmatrix}_{a;q}=0,\qquad\text{for
$k>n$ or $k<0$}, \end{equation*}
determine them uniquely.
\end{remark}


\section{$(a;q)$-Rook theory}
\label{sec:rooks}
For an introduction to classical rook theory, see \cite{BCHR}.
A lot of material on generalized rook theory which we survey
is borrowed from our papers
\cite{SY0,SY1,SY2} on elliptic rook theory. As mentioned in the introduction,
the $(a;q)$-case is just a special case of the elliptic case which admits 
particularly attractive closed formulas. We utilize the closed formulas 
from the $(a;q)$-extensions of different rook models
to derive some concrete basic hypergeometric summations. 


\subsection{$(a;q)$-Extension of the standard model}

Let $\N$ denote the set of positive integers. We consider a finite subset 
of the $\N \times \N$ grid which we refer to as a \emph{board}, 
and label the columns and rows by $1, 2, \dots$, from the left
and from the bottom, respectively. We use $(i,j)$ to denote the cell in the 
intersection of the column $i$ and the row $j$.

Let $B(b_1, b_2,\dots, b_n)$ denote the set of cells 
\begin{displaymath}
B=B(b_1,\dots, b_n)=\{(i,j)~|~ 1\le i\le n,~ 1 \le j\le b_i\},
\end{displaymath}
for nonnegative integer $b_i$'s, for all $i$. 
If a board $B$ can be represented by the set $B(b_1,\dots, b_n)$ with 
nondecreasing integer sequence $0\le b_1\le \cdots \le b_n$,
then the board $B=B(b_1,\dots, b_n)$ is called a \emph{Ferrers board}.
Given a Ferrers board, we say that we \emph{place $k$ nonattacking rooks}
in $B$ by choosing a $k$-subset of $B$ such that no two elements
have a common coordinate.
Let $\mathcal{N}_{k}(B)$ denote the set of all nonattacking 
placements of $k$ rooks in $B$. 
Note that $|\mathcal{N}_{k}(B)|$ is the original $k$-th rook number
defined in \cite{KR}.

Given a rook placement $P\in\mathcal{N}_{k}(B)$,
a rook in $P$ is said to {\em cancel} all the cells to the right in the same row
and  all the cells below it in the same column. Let $U_B(P)$ denote the 
set of cells in $B-P$ which are not cancelled by any rook in $P$. 
We define the $(a;q)$-analogue of the $k$-th rook number by assigning 
the small weights $w_{a;q}(j)$ to the respective cells in $U_B(P)$,
depending on their position and the configuration of rooks.

\begin{definition}
Given a Ferrers board $B=B(b_1,\dots, b_n)$, let the $k$th
$(a;q)$-rook number be
\begin{equation}\label{eqn:aqrook}
r_k (a,q;B)=\sum_{P\in \mathcal N_k(B)}wt(P), 
\end{equation}
where 
$$wt(P)=\prod_{(i,j)\in U_B(P)}w_{a;q}(i-j-r_{(i,j)}(P)),$$
and $r_{(i,j)}(P)$ counts the number of rooks in $P$ positioned in
the north-west region of $(i,j)$.
\end{definition}

This $(a;q)$-analogue of the rook numbers satisfy the following product 
formula which was proved with original rook numbers, i.e., in the
$a\to\infty$, $q\to 1$ case, by Goldman, Joichi,
and White \cite{GJW4}. 

\begin{theorem}[\cite{SY0}]\label{thm:aqprod}
For any Ferrers board $B=B(b_1,\dots, b_n)$, we have 
\begin{equation*}
\prod_{i=1}^n [z+b_i-i+1]_{aq^{2(i-1-b_i)};q}=
\sum_{k=0}^n r_{n-k}(a,q;B)\prod_{j=1}^k [z-j+1]_{aq^{2(j-1)};q}.
\end{equation*}
\end{theorem}

\begin{remark}
In \cite{SY0}, we prove Theorem \ref{thm:aqprod} with more general
rook numbers. That is, the rook numbers $r_k (a,b;q,p;B)$
and the weights used to define $r_k$ in \eqref{eqn:aqrook} are elliptic
(i.e., meromorphic and doubly-periodic) and include two more parameters
$b$ and $p$. The
$(a;q)$-rook numbers can be obtained from the 
elliptic ones by letting $p\to 0$ and $b\to 0$. 
\end{remark}

By distinguishing the cases when there is a rook or not in the last column, 
we obtain a recursion for the $(a;q)$-rook numbers. 

\begin{proposition}\label{prop:recur}
Let $B$ be a Ferrers board with $l$ columns of height at most $m$,
and $B\cup m$ denote the board obtained by adding the $(l+1)$-st
column of height $m$
to the right of $B$. Then, for integer $k$ with $1\le k \le l+1$, we have 
\begin{equation*}
r_k(a,q;B\cup m) =W_{aq^{2(l-m)};q}(m-k)r_k (a,q;B)+[m-k+1]_{aq^{2(l-m)};q}
r_{k-1}(a,q;B),
\end{equation*}
assuming the conditions 
\begin{align*}
r_k (a,b;q,p;B)={}&0\qquad\text{for $k<0$ or $k>l$, and }\\
r_0 (a,b;q,p;B)={}&1\qquad\text{for $l=0$, i.e.\ for
$B$ being the empty board}.
\end{align*}
\end{proposition}

In the case of a rectangular shape board $B=[l]\times [m]$,
where $[n]:=\{1,2, \dots, n\}$, the $(a;q)$-rook number has a closed
form expression which can be proved by the recursion in 
Proposition~\ref{prop:recur}:
\begin{equation}\label{eqn:aqrect}
r_k(a,q;[l]\times[m])=
q^{\binom{k+1}{2}-lm}\begin{bmatrix}l\\k\end{bmatrix}_{q}
\frac{[m]_q !}{[m-k]_q !}
\frac{(aq^{l-m-k} ;q)_{k}(aq^{1+2l-2m};q^2)_{m-k}}{(aq^{1-2m};q^2)_{m}}. 
\end{equation}
For more details, including the omitted proofs, see \cite{SY0}.


\subsubsection{$r$-Restricted Lah numbers}


The $r$-restricted Lah numbers count the number of
placements of the elements $1,2,\dots,n$ into $k$ nonempty tubes
of linearly ordered elements such that $1,2,\dots,r$ are in distinct tubes
(cf.\ \cite{NR}, or \cite{MS}). These numbers admit a rook theoretic 
interpretation when $B$ is the board $\mathsf L^{(r)}_n=[n+r-1]\times[n-r]$.
In \cite[Subsection 3.4]{SY0}, we have established a correspondence 
between the rook configurations $P$ of $n-k$ nonattacking rooks 
on $\mathsf L^{(r)}_n$ and the set of placements $T$
of the elements  $1,2,\dots,n$ into $k$ nonempty tubes of linearly ordered
elements such that the first $r$ numbers $1,2,\dots,r$ are in distinct tubes. 
For the full description of the correspondence, refer to \cite{SY0}.

For the Ferrers board $B=\mathsf L^{(r)}_n$, the product formula in 
Theorem \ref{thm:aqprod} becomes
\begin{align}\label{eqn:aqLahprod}
&\prod_{i=1}^{n-r} [z+n-i]_{aq^{2(i-n)};q}
\prod_{i=1}^r [z-i+1]_{aq^{2(i-1)};q}\notag\\
&=\sum_{k=r}^{n}r_{n-k}(aq^{2(1-r)},q;\mathsf L^{(r)}_n)
\prod_{j=1}^k [z-j+1]_{aq^{2(j-1)};q},
\end{align}
after doing certain shifts of variables and cancellation of factors.  
We define an $(a;q)$-analogue of the $r$-restricted Lah numbers by
$$\mathcal L_{n,k}^{(r)}(a,q):= r_{n-k}(aq^{2(1-r)},q;\mathsf L^{(r)}_n).$$
It can be shown that $\mathcal L_{n,k}^{(r)}(a,q)$ satisfy
the following recursion
\begin{equation*}\label{eqn:recellrlah}
\mathcal{L}^{(r)}_{n+1,k}(a,q)=
W_{aq^{-2n};q}(n+k-1)\,\mathcal{L}^{(r)}_{n, k-1}(a,q)
+[n+k]_{aq^{-2n};q}\,\mathcal{L}^{(r)}_{n,k}(a,q),
\end{equation*}
assuming the initial conditions
\begin{align*}
\mathcal{L}^{(r)}_{n,k}(a,q)&=0\qquad\text{for $k<r-1$ or $k>n$},\notag\\
\mathcal{L}^{(r)}_{r-1,r-1}(a,q)&=1\qquad\text{(an artificial but
felicious initial condition)}.
\end{align*}

Since the board $\mathsf L^{(r)}_n$ is of rectangular shape,
\eqref{eqn:aqrect} 
gives the closed form formula for $\mathcal L_{n,k}^{(r)}(a,q)$, namely,
\begin{equation}\label{eqn:aqLah}
\mathcal{L}^{(r)}_{n,k}(a,q)=
q^{\binom k2-\binom n2-n(k-1)+2\binom r2}
\begin{bmatrix}n+r-1\\k+r-1\end{bmatrix}_{q}
\frac{[n-r]_q !}{[k-r]_q !}
\frac{(aq^{1-n+k} ;q)_{n-k}(aq^{1+2r};q^2)_{k-r}}{(aq^{3-2n};q^2)_{n-r}}.
\end{equation}
Combining \eqref{eqn:aqLah} with the product formula
\eqref{eqn:aqLahprod} gives 
a combinatorial proof of the $q$-Pfaff-Saalsch\"utz sum, in the following form :

\begin{proposition}
\begin{equation}\label{eqn:qPfaff}
\frac{(q^{z+r};q)_n (a^{-1}q^{r-z};q)_n}{(a^{-1};q)_n (q^{2r};q)_n}
=\sum_{k=0}^n \frac{(q^{-n};q)_k (q^{r-z};q)_k (aq^{z+r};q)_k}
{(q;q)_k (q^{2r};q)_k (aq^{1-n};q)_k}q^k.
\end{equation}
\end{proposition}

\begin{proof}
If we replace $r_{n-k}(aq^{2(1-r)},q;\mathsf L^{(r)}_n)$ by the
closed form given in \eqref{eqn:aqLah}, we obtain
\begin{multline}\label{eqn:qPfaffpf}
\frac{(q^{z+r};q)_{n-r}(aq^{z-n+1};q)_{n-r}}{(1-q)^{n-r}(aq^{3-2n};q^2)_{n-r}}
\frac{(q^{z-r+1};q)_r(aq^z ;q)_r}{(1-q)^r(aq;q^2)_r}q^{-nz+\frac{1}{2}n(3-n)+r^2-r}\\
=\sum_{k=r}^n q^{\binom k2-\binom n2-n(k-1)+2\binom r2-kz+\binom{k+1}{2}}
\frac{(q;q)_{n+r-1}}{(q;q)_{n-k}(q;q)_{k+r-1}}
\frac{(q;q)_{n-r}}{(q;q)_{k-r}(1-q)^{n-k}}\\
\times \frac{(q^{z-k+1};q)_k (aq^z ;q)_k}{(1-q)^k (aq;q^2)_k}
\frac{(aq^{1-n+k} ;q)_{n-k}(aq^{1+2r};q^2)_{k-r}}{(aq^{3-2n};q^2)_{n-r}}.
\end{multline}
Then \eqref{eqn:qPfaff} is the result of simplifying
\eqref{eqn:qPfaffpf} with appropriate shifts of $n$ and $k$.
\end{proof}

\begin{remark}
If we perform the substitution $A=aq^{z+r}$, $B=q^{r-z}$ and $C=q^{2r}$ in
\eqref{eqn:qPfaff}, we get 
\begin{equation}\label{eqn:qPfaffor}
\frac{(C/A, C/B;q)_n}{(C, C/AB;q)_n}=
{}_3\phi_2\left[\begin{array}{c}A,B,q^{-n}\\
C, ABC^{-1}q^{1-n}\end{array};q,q \right]
\end{equation}
which is the $q$-Pfaff-Saalsch\"utz summation \cite[(II.12)]{GR}, written in
standard basic hypergeometric form (cf.\ \cite{GR}).
The problem with this substitution is that whereas $a$ and $z$ are general
parameters, $r$ is not.
To show that \eqref{eqn:qPfaff}, where $r$ is a nonnegative integer,
 is actually equivalent to the general case where $r$ is any complex number,
 works by a standard polynomial argument.
 If we multiply both sides of \eqref{eqn:qPfaff} by $(q^{2r};q)_n$
 and formally replace $q^r$ by $x$ we obtain a polynomial equation in $x$ of
 degree $2n$ which is valid for $x=q^r$, for $r=0,1,2\dots$
(i.e., for more than $2n$ values) thus must be true for all complex $x$.
\end{remark}

As mentioned in the introduction, there exist also other combinatorial proofs
of the ${}_3\phi_2$ summation. Among the references we have listed,
Yee's paper \cite{Y} is remarkable as the proof there establishes the full
$q$-Pfaff-Saalsch\"utz summation at once and no appeal to a polynomial argument
is needed.


\subsection{$(a;q)$-Extension of the alpha-parameter model}


In \cite{GH}, Goldman and Haglund introduced generalized rook models,
called \emph{$i$-creation model} and \emph{alpha-parameter model},
which we briefly introduce first. 

Given a board $B$, a \emph{file placement} of $k$ rooks is a
$k$-subset of $B$ such that no two cells lie in the same column,
that is, there can be two or more rooks in the same row,
but each column contains at most one rook.
Let $\mathcal F_k(B)$ denote the set of all $k$-file placements.
Given a Ferrers board $B$ and a file placement
$P\in \mathcal F_k (B)$, 
we assign weights to the rows containing rooks as follows.
If there are $u$ rooks in a given 
row, then the weight of this row is 
\begin{equation*}
\begin{cases}
1 & \text{ if } 0\le u\le 1,\\
\alpha(2 \alpha-1)(3\alpha -2)\cdots ((u-1)\alpha -(u-2)),& 
\text{ if } u\ge 2. \end{cases}
\end{equation*}
The weight of a placement $P$, $wt(P)$, is the product of the weights of
all the rows. Then for a Ferrers board $B$, set 
$$r_k ^{(\alpha)}(B)=\sum_{P\in \mathcal F_k (B)}wt(P).$$
Note that for $\alpha =0$, $r_k ^{(0)}(B)$ reduces to the original
rook number. If $\alpha$ is a positive integer $i$, $r_k ^{(i)}(B)$
is the \emph{$i$-creation}
rook number which counts the number of $i$-creation rook placements of
$k$ rooks on $B$. The $i$-creation rook placement is defined as follows:
we first choose the columns to place the rooks. Then as we place rooks
from left to right, each time a rook is placed, $i$ new rows are created
drawn to the right end and immediately above where the rook is placed. 

In this setting Goldman and Haglund \cite{GH} proved the
\emph{$\alpha$-factorization theorem}:
given a Ferrers board $B=B(b_1,\dots, b_n)$,
\begin{equation*}\label{eqn:alpha}
\prod_{j=1}^n (z+b_j +(j-1)(\alpha -1)) 
= \sum_{k=0}^n r_k ^{(\alpha)}(B)z (z+\alpha -1)\cdots (z+(n-k-1)(\alpha -1)).
\end{equation*}
Furthermore, Goldman and Haglund defined a $q$-analogue of
$r_k ^{(\alpha)}(B)$ by assigning $q$-weights to the cells in $B$.
Here, we describe the $(a;q)$-extension of their result 
which involves the use of the extra variable $a$ in the weights of the cells. 

Given a Ferrers board $B$ and a rook placement $P\in \mathcal N_k(B)$,
for each cell $c\in B$, let $v(c)$ be
the number of rooks strictly to the left of, and in the same row as $c$
and $r_c(P)$ be the number of rooks in the north-west region 
of $c$. Then define the weight of $c$ to be 
\begin{equation*}
wt_\alpha(c)=
\begin{cases}
 1, \hfill\text{ if there is a rook above and in the same column as $c$,}\\
 [(\alpha -1)v(c)+1]_{aq^{2(-j+(\alpha-1)(1-i+r_c(P)))};q},
\qquad\qquad \text{ if $c$ contains a rook,}\\
 W_{aq^{2(-j+(\alpha-1)(1-i+r_c(P)))};q}((\alpha -1)v(c) +1),
\hfill\text{ otherwise.}\qquad\qquad
\end{cases}
\end{equation*}
The weight of the rook placement $P$ is defined to be the product of 
the weights of all cells:
\begin{equation*}
wt_\alpha (P)=\prod_{c\in B}wt_\alpha (c).
\end{equation*}
We define an $(a;q)$-analogue of $r_k ^{(\alpha)}(B)$ by setting 
\begin{equation*}
r_k ^{(\alpha)}(a,q;B)=\sum_{P\in\mathcal F_k (B) }wt_\alpha (P).
\end{equation*}
With this $r_k ^{(\alpha)}(a,q;B)$, we can also prove an $(a;q)$-analogue of 
the $\alpha$-factorization theorem.

\begin{theorem}\label{thm:aqalpha}
For any Ferrers board $B=B(b_1,b_2,\dots, b_n)$, we have 
\begin{align}\label{eqn:aqalphaprod}
&\prod_{j=1}^n [z+b_j+(j-1)(\alpha -1)]_{aq^{-2(b_j+(j-1)(\alpha-1))};q}\notag\\
&=\sum_{k=0}^n r_{n-k}^{(\alpha)} (a,q;B)
\prod_{i=1}^k[z+(i-1)(\alpha -1)]_{aq^{-2(i-1)(\alpha-1)};q}. 
\end{align}
\end{theorem}

\begin{proof}
Let us extend the board by attaching $z$ rows of width $n$ below the
board $B$, denoted by $B_z$, and compute 
$$\sum_{P\in\mathcal F_n (B_z)}wt_\alpha (P)$$
in two different ways. The left-hand side of \eqref{eqn:aqalphaprod}
is the result of computing the above weight sum columnwise, and the
right-hand side can be obtained by computing the weight of the cells
in $B$ and and the cells in the extended part 
separately. For the details, see \cite{SY2}.
\end{proof}

\begin{remark}
In \cite{SY1}, the authors have constructed a general rook theory model 
utilizing an augmented rook board
which can be specialized to all the known rook theory models.
The product formula in Theorem \ref{thm:aqalpha} was also obtained in
\cite[(4.17)]{SY1} but by using a different approach.
\end{remark}

In the case $\alpha=2$ and the board is of the staircase shape
$St_n=B(0,1,2,\dots, n-1)$, 
$r_k ^{(2)}(a,q;St_n)$ has a closed form expression (see \cite{SY0}).
\begin{equation}\label{eqn:alpha2rk}
r_k ^{(2)}(a,q;St_n)
=q^{-\binom{n+k}{2}+k(k+2)}\begin{bmatrix}n+k-1\\2k\end{bmatrix}_q
\prod_{j=1}^k [2j-1]_q
\frac{(a q;q^{-2})_{n-k} (a q^{1-2n};q^2)_k}{(a q;q^{-4})_n}.
\end{equation}

We now give combinatorial proofs of two special
${}_4\phi_3$ summations.

\begin{proposition}
\begin{equation}\label{eqn:aqalphaprop1}
\frac{(q^{z+2},a^{-1}q^{2-z};q^2)_n}{(q,a^{-1}q^3;q^2)_n} 
=\sum_{k=0}^n \frac{(q^{-n}, -q^{-n},q^{z+1},a^{-1}q^{1-z};q)_k}
{(q,q^{-2n},a^{-1/2}q^{3/2},-a^{-1/2}q^{3/2};q)_k}q^k,
\end{equation}
and,
\begin{equation}\label{eqn:aqalphaprop2}
\frac{(q^{z+2},aq^{z-2n};q^2)_n}{(q^{z+1},aq^{z-n};q)_n} 
=\sum_{k=0}^n \frac{(q^{-n}, q^{n+1}, a^{1/2}q^{-n-1/2}, -a^{1/2}q^{-n-1/2};q)_k}
{(q,-q, q^{-z-n}, aq^{z-n};q)_k}q^k.
\end{equation}
\end{proposition}

\begin{proof}
If we use the closed form expression for $r_{n-k}^{(2)}(a,q;St_n)$
of \eqref{eqn:alpha2rk}
in \eqref{eqn:aqalphaprod} for $b_j=j-1$, we get 
\begin{multline*}
\frac{(q^z;q^2)_n}{(1-q)^n}\frac{(a q^z;q^{-2})_n}
{(aq;q^{-4})_n}q^{-n(z+1)}=\\
\sum_{k=0}^n q^{-k(z+1)}\frac{(q;q)_{2n-k-1}}{(q;q)_{2n-2k}(q;q)_{k-1}}
\frac{(q;q^2)_{n-k}}{(1-q)^{n-k}}
\frac{(aq;q^{-2})_k (aq^{1-2n};q^2)_{n-k}}{(aq;q^{-4})_n}
\frac{(q^z ;q)_k}{(1-q)^k}\frac{(aq^z;q^{-1})_k}{(aq;q^{-2})_k}
\end{multline*}
which after some elementary manipulations simplifies to
\eqref{eqn:aqalphaprop1}.

Similarly, \eqref{eqn:aqalphaprop2} is 
the result of simplifying 
\begin{equation*}
\prod_{j=1}^n [z+2(j-1)]_{aq^{-4(j-1)};q}
=\sum_{k=0}^n r_{k}^{(2)} (a,q;St_n)
\prod_{i=1}^{n-k}[z+i-1]_{aq^{-2(i-1)};q},
\end{equation*}
after replacing $r_k ^{(2)}(a,q;St_n)$ by the closed form expression
in \eqref{eqn:alpha2rk}.
\end{proof}
The summation in \eqref{eqn:aqalphaprop1}
is equivalent (up to an obvious substitution of variables)
to Jain's ${}_4\phi_3$ summation \eqref{43id}
mentioned in the introduction.

The summation in \eqref{eqn:aqalphaprop2}
can be also verified by the following terminating $q$-analogue 
of Whipple's ${}_3 F_2$ sum \cite[(II.19)]{GR},
\begin{equation*}
{}_4\phi_3 \left[\begin{array}{c} q^{-n}, q^{n+1}, C, -C \\ 
E, C^2 q/E, -q \end{array};q,q \right]
=\frac{(E q^{-n},E q^{n+1}, C^2 q^{1-n}/E, C^2 q^{n+2}/E;q^2)_{\infty}}
{(E, C^2 q/E;q)_\infty}q^{\binom{n+1}{2}},
\end{equation*}
where we take $C=a^{1/2}q^{-n-1/2}$ and $E=q^{-z-n}$, and apply 
\begin{align*}
&\frac{(q^{-z-2n},q^{1-z},aq^{z-2n},aq^{z+1};q^2)_\infty}
{(q^{-z-n},aq^{z-n};q)_\infty}q^{\binom{n+1}{2}}\\
=&\frac{(q^{-z-2n};q^2)_n}{(q^{-z-n};q)_n}
\frac{(q^{-z},q^{1-z};q^2)_\infty}{(q^{-z};q)_\infty}
\frac{(aq^{z-2n};q^2)_n}{(aq^{z-n};q)_n}
\frac{(aq^z,aq^{z+1};q^2)_\infty}{(aq^z ;q)_\infty}q^{\binom{n+1}{2}}\\
=&\frac{(q^{z+2};q^2)_n}{(q^{z+1};q)_n}\frac{(aq^{z-2n};q^2)_n}{(aq^{z-n};q)_n}.
\end{align*}
The two summations in \eqref{eqn:aqalphaprop1} and
\eqref{eqn:aqalphaprop2} are actually equivalent to each other;
one follows from the other by reversing the sum (i.e., substituting
the summation index $k\mapsto n-k$).


\subsection{$(a;q)$-Rook theory for matchings}


Haglund and Remmel \cite{HR} extended the rook theory by considering 
partial matchings as opposed to considering partial permutations 
in the original rook theory, and for which they consider the shifted board 
$B_{2n}$ pictured in Figure~\ref{fig:B2n}.

\setlength{\unitlength}{1.2pt}
\begin{figure}[ht]
\begin{picture}(120,116)(0,0)
\multiput(105,0)(0,15){1}{\line(1,0){15}}
\multiput(90,15)(0,15){1}{\line(1,0){30}}
\multiput(75,30)(0,15){1}{\line(1,0){45}}
\multiput(60,45)(0,15){1}{\line(1,0){60}}
\multiput(45,60)(0,15){1}{\line(1,0){75}}
\multiput(30,75)(0,15){1}{\line(1,0){90}}
\multiput(15,90)(0,15){1}{\line(1,0){105}}
\multiput(15,105)(0,15){1}{\line(1,0){105}}
\multiput(15,90)(0,15){1}{\line(0,1){15}}
\multiput(30,75)(0,15){1}{\line(0,1){30}}
\multiput(45,60)(0,15){1}{\line(0,1){45}}
\multiput(60,45)(0,15){1}{\line(0,1){60}}
\multiput(75,30)(0,15){1}{\line(0,1){75}}
\multiput(90,15)(0,15){1}{\line(0,1){90}}
\multiput(105,0)(0,15){1}{\line(0,1){105}}
\multiput(120,0)(0,15){1}{\line(0,1){105}}
\put(20,107){$2$}
\put(35,107){$3$}
\put(50,109){$\cdot$}
\put(64,109){$\cdot$}
\put(78,109){$\cdot$}
\put(88,107){$2n$-$1$}
\put(109,107){$2n$}
\put(6, 94){$1$}
\put(6, 79){$2$}
\put(7, 63){$\cdot$}
\put(7, 48){$\cdot$}
\put(7, 33){$\cdot$}
\put(1,17){$2n$-$2$}
\put(1, 2){$2n$-$1$}
\end{picture}
\caption{$B_{2n}$.}\label{fig:B2n}
\end{figure}

For each perfect matching $M$ of $K_{2n}$ consisting of
$n$ pairwise vertex disjoint edges in $K_{2n}$,
where $K_{2n}$ is the complete graph on the set of vertices
$\{1,2,\dots, 2n\}$, let 
$$P_{M}=\{ (i,j) ~|~ i<j \text{ and } \{ i,j\} \in M\},$$
where $(i,j)$ denotes the square in row $i$ and column $j$ of $B_{2n}$
according to the labeling of rows and columns pictured in
Figure~\ref{fig:B2n}. A rook placement in $B_{2n}$ is defined to
be a subset of some $P_M$ for a perfect matching $M$ of $K_{2n}$. 

Given a board $B\subseteq B_{2n}$,
we let $\mathcal{M}_k (B)$ denote the set of $k$ element rook placements
in $B$. In this setting, we let $B(a_1,a_2,\dots, a_{2n-1})$ denote the 
following set of cells in $B_{2n}$ :
$$B(a_1,a_2,\dots, a_{2n-1})=\{ (i,i+j)~|~ 1\le i\le 2n-1,1\le j\le a_i\}.$$
It is called a \emph{shifted Ferrers board} if
$2n-1\ge a_1\ge a_2 \ge \cdots \ge a_{2n-1} \ge 0$ and the nonzero entries
of $a_i$'s are strictly decreasing. A rook in $(i,j)$ with $i<j$ in a
rook placement \emph{cancels} all cells $(i,s)$ in $B_{2n}$ with $i<s<j$ and
all cells $(t,j)$ and $(t,i)$ with $t<i$.
See Figure~\ref{fig:sboardc} for a specific example of a shifted Ferrers board
and the cells being cancelled by a rook on the shifted board $B_8$.\\
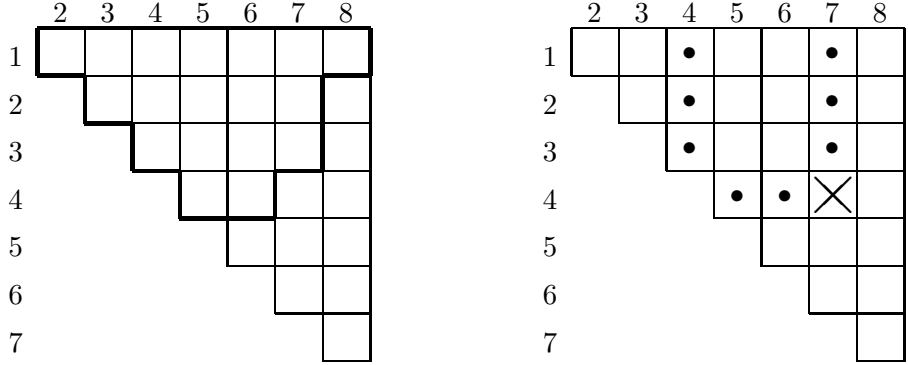
\begin{figure}[ht]
$$\begin{array}{cc}
\begin{picture}(160,123)(0,-3)
\multiput(105,0)(0,15){1}{\line(1,0){15}}
\multiput(90,15)(0,15){1}{\line(1,0){30}}
\multiput(75,30)(0,15){1}{\line(1,0){45}}
\multiput(60,45)(0,15){1}{\line(1,0){60}}
\multiput(45,60)(0,15){1}{\line(1,0){75}}
\multiput(30,75)(0,15){1}{\line(1,0){90}}
\multiput(15,90)(0,15){1}{\line(1,0){105}}
\multiput(15,105)(0,15){1}{\line(1,0){105}}
\multiput(15,90)(0,15){1}{\line(0,1){15}}
\multiput(30,75)(0,15){1}{\line(0,1){30}}
\multiput(45,60)(0,15){1}{\line(0,1){45}}
\multiput(60,45)(0,15){1}{\line(0,1){60}}
\multiput(75,30)(0,15){1}{\line(0,1){75}}
\multiput(90,15)(0,15){1}{\line(0,1){90}}
\multiput(105,0)(0,15){1}{\line(0,1){105}}
\multiput(120,0)(0,15){1}{\line(0,1){105}}
\thicklines\linethickness{1.3pt}
\multiput(15,105)(0,15){1}{\line(1,0){105}}
\multiput(15,90)(0,15){1}{\line(0,1){15}}
\multiput(15,90)(0,15){1}{\line(1,0){15}}
\multiput(30,75)(0,15){1}{\line(0,1){15}}
\multiput(30,75)(0,15){1}{\line(1,0){15}}
\multiput(45,60)(0,15){1}{\line(0,1){15}}
\multiput(45,60)(0,15){1}{\line(1,0){15}}
\multiput(60,45)(0,15){1}{\line(0,1){15}}
\multiput(60,45)(0,15){1}{\line(1,0){30}}
\multiput(90,60)(0,15){1}{\line(1,0){15}}
\multiput(105,90)(0,15){1}{\line(1,0){15}}
\multiput(90,45)(0,15){1}{\line(0,1){15}}
\multiput(105,60)(0,15){1}{\line(0,1){30}}
\multiput(120,90)(0,15){1}{\line(0,1){15}}
\put(20,107){$2$}
\put(35,107){$3$}
\put(50,107){$4$}
\put(65,107){$5$}
\put(80,107){$6$}
\put(95,107){$7$}
\put(110,107){$8$}
\put(6, 93){$1$}
\put(6, 78){$2$}
\put(6, 63){$3$}
\put(6, 48){$4$}
\put(6, 33){$5$}
\put(6, 18){$6$}
\put(6, 3){$7$}
\end{picture}
&
\begin{picture}(120,123)(0,-3)
\multiput(105,0)(0,15){1}{\line(1,0){15}}
\multiput(90,15)(0,15){1}{\line(1,0){30}}
\multiput(75,30)(0,15){1}{\line(1,0){45}}
\multiput(60,45)(0,15){1}{\line(1,0){60}}
\multiput(45,60)(0,15){1}{\line(1,0){75}}
\multiput(30,75)(0,15){1}{\line(1,0){90}}
\multiput(15,90)(0,15){1}{\line(1,0){105}}
\multiput(15,105)(0,15){1}{\line(1,0){105}}
\multiput(15,90)(0,15){1}{\line(0,1){15}}
\multiput(30,75)(0,15){1}{\line(0,1){30}}
\multiput(45,60)(0,15){1}{\line(0,1){45}}
\multiput(60,45)(0,15){1}{\line(0,1){60}}
\multiput(75,30)(0,15){1}{\line(0,1){75}}
\multiput(90,15)(0,15){1}{\line(0,1){90}}
\multiput(105,0)(0,15){1}{\line(0,1){105}}
\multiput(120,0)(0,15){1}{\line(0,1){105}}
\thicklines
\multiput(92,58)(0,15){1}{\line(1,-1){11}}
\multiput(92,47)(0,15){1}{\line(1,1){11}}
\put(20,107){$2$}
\put(35,107){$3$}
\put(50,107){$4$}
\put(65,107){$5$}
\put(80,107){$6$}
\put(95,107){$7$}
\put(110,107){$8$}
\put(6, 93){$1$}
\put(6, 78){$2$}
\put(6, 63){$3$}
\put(6, 48){$4$}
\put(6, 33){$5$}
\put(6, 18){$6$}
\put(6, 3){$7$}
\put(95,95){$\bullet$}
\put(95,80){$\bullet$}
\put(95,65){$\bullet$}
\put(80,50){$\bullet$}
\put(65,50){$\bullet$}
\put(50,95){$\bullet$}
\put(50,80){$\bullet$}
\put(50,65){$\bullet$}
\end{picture}
\end{array}$$
\caption{The shifted Ferrers board $B=(7,5,4,2,0,0,0)\subseteq B_8$, 
and the cells cancelled by a rook in $(4,7)$ on $B_8$.}\label{fig:sboardc}
\end{figure}

\begin{definition}
Given a shifted Ferrers board $B=B(a_1,\dots, a_{2n-1})\subseteq B_{2n}$
and a rook placement $P\in \mathcal{M}_k (B)$, define
\begin{equation*}
m_k (a,q;B)=\sum_{P\in \mathcal M_k (P)}wt_m(P), 
\end{equation*}
where 
\begin{equation*}
wt_m (P)=\prod_{(i,j)\in U_B (P)}
w_{a;q}(\hat{i}+\hat{j}-1-2r_{(i,j)}(P)-s_{(i,j)}(P)),
\end{equation*}
$U_{B}(P)$ denote
the set of cells in $B$ which are neither cancelled by rooks
nor contain any rooks 
in $P$, $r_{(i,j)}(P)$ is the number of rooks in $P$ positioned south-east
of $(i,j)$ such that the two columns cancelled by those rooks are
to the right of the column $j$, $s_{(i,j)}(P)$ is the number of rooks
in $P$ which are in the south-east region of $(i,j)$ such that only
one cancelled column
is to the right of column $j$, and $\hat{i}:=2n-i$.
\end{definition}

We also have a product formula involving $m_k (a,q;B)$, which
is an $(a;q)$-analogue 
of the product formula proved by Haglund and Remmel~\cite{HR}.

\begin{theorem}
Given a shifted Ferrers board $B=B(a_1,\dots, a_{2n-1})\subseteq B_{2n}$, 
we have 
\begin{equation}\label{eqn:mprod}
\prod_{i=1}^{2n-1}[z+a_{2n-i}-2i+2]_{aq^{2(2i-2-a_{2n-i})};q}
=\sum_{k=0}^n m_k (a,q;B)\prod_{j=1}^{2n-1-k}[z-2j+2]_{aq^{4j-4};q}.
\end{equation} 
\end{theorem}

In the case of the full board $B_{2n}=B(2n-1, 2n-2, \dots, 2,1)$,
$m_k(a,q;B_{2n})$ has a closed form 
\begin{equation}\label{eqn:mkB2n}
m_k (a,q;B_{2n})=q^{k^2 -\binom{2n}{2}}\begin{bmatrix}2n\\2k\end{bmatrix}_q 
\prod_{j=1}^k [2j-1]_q \frac{(aq^{4n-2k-3};q^2)_{2n-k-1}}{(aq^{-1};q^4)_{2n-k-1}},
\end{equation}
which can be verified by the following recursion
\begin{align*}\label{eqn:mkrecur}
 &m_k(a,q;B_N)\\
&= [N-2k+1]_{aq^{2(N-3)};q}\,m_{k-1}(a,q;B_{N-1})+W_{aq^{2(N-3)};q}(N-2k-1)\,m_k(a,q;B_{N-1}).
\end{align*}
Replacing $m_k(a,q;B)$ in \eqref{eqn:mprod} for $a_i=2n-i$ by \eqref{eqn:mkB2n}
gives a special case of the $q$-Pfaff-Saalsch\"utz sum. 

\begin{proposition}
\begin{equation}\label{prop:mk}
 \frac{(q^{z-n+\frac{3}{2}},aq^{z-\frac{1}{2}};q)_n}
{(q^{z-2n+2},aq^{z+n-\frac{1}{2}};q)_n}
 =\sum_{k=0}^n \frac{(q^{-n},q^{-n+\frac{1}{2}},q^{\frac{5}{2}-2n}/a;q)_k}
{(q,q^{z-2n+2},q^{2-2n-z}/a;q)_k}q^k.
\end{equation} 
\end{proposition}

\begin{proof}
Putting the closed form expression for $m_k(a,q;B_{2n})$ in
\eqref{eqn:mprod} for $a_i=2n-i$ gives 
\begin{multline*}
 \frac{(q^{z+1};q^{-1})_{2n-1}(aq^{z-1};q)_{2n-1}}
{(1-q)^{2n-1}(aq^{-1};q^2)_{2n-1}}\\
 =\sum_{k=0}^n q^{2k(k-2n+1)+kz}\frac{(q;q)_{2n}}{(q;q)_{2k}(q;q)_{2n-2k}}
\frac{(q;q^2)_k}{(1-q)^k}\frac{(aq^z,aq^{4n-2k-3};q^2)_{2n-k-1}}
{(aq,aq^{-1};q^4)_{2n-k-1}}
 \frac{(q^z;q^{-2})_{2n-k-1}}{(1-q)^{2n-k-1}}
\end{multline*}
which simplifies to 
\begin{equation*}
\frac{(q^{-z-2},aq^{z-2};q)_{2n}}{(q^{-z-2},aq^{z-2};q^2)_{2n}}q^{n(2n-1)}
=\sum_{k=0}^n \frac{(q^{-2n},q^{1-2n},q^{5-4n}/a;q^2)_k}
{(q^2, q^{z-4n+4},q^{4-4n-z}/a;q^2)_k}q^{2k}.
\end{equation*}
The identitiy can now
be obtained by replacing $q^2\to q$ and $z/2\to z$. 
\end{proof}
\begin{remark}
The identity \eqref{prop:mk} (proved combinatorially) is actually the
\begin{equation*}
A=q^{-n+\frac{1}{2}},\quad B=q^{\frac{5}{2}-2n}/a,
\quad\text{ and } ~C=q^{z-2n+2}
\end{equation*}
special case of the
$q$-Pfaff-Saalsch\"utz sum \eqref{eqn:qPfaffor}.
\end{remark}


 
\end{document}